\documentclass[12pt,a4paper]{amsart}
\makeatletter
\renewcommand\normalsize{%
    \@setfontsize\normalsize{11.7}{14pt plus .3pt minus .3pt}%
    \abovedisplayskip 10\p@ \@plus4\p@ \@minus4\p@
    \abovedisplayshortskip 6\p@ \@plus2\p@
    \belowdisplayshortskip 6\p@ \@plus2\p@
    \belowdisplayskip \abovedisplayskip}
\renewcommand\small{%
    \@setfontsize\small{9.5}{12\p@ plus .2\p@ minus .2\p@}%
    \abovedisplayskip 8.5\p@ \@plus4\p@ \@minus1\p@
    \belowdisplayskip \abovedisplayskip
    \abovedisplayshortskip \abovedisplayskip
    \belowdisplayshortskip \abovedisplayskip}
\renewcommand\footnotesize{%
    \@setfontsize\footnotesize{8.5}{9.25\p@ plus .1pt minus .1pt}%%
    \abovedisplayskip 6\p@ \@plus4\p@ \@minus1\p@
    \belowdisplayskip \abovedisplayskip
    \abovedisplayshortskip \abovedisplayskip
    \belowdisplayshortskip \abovedisplayskip}
\ifdefined\pdfpagewidth
\else
\fi
\calclayout
\makeatother

\usepackage{graphicx,calc}
\usepackage{amsmath,amsfonts,amssymb,amsthm, color}
\usepackage{pstricks} 
\usepackage{epstopdf}
\usepackage{srcltx}
\usepackage{comment}
\usepackage{hyperref}
\usepackage[normalem]{ulem}
\usepackage[all]{xy}
\usepackage[utf8]{inputenc} 
\usepackage[T1]{fontenc} 

\usepackage[all]{xy}
\usepackage{bbm}

\definecolor{cerulean}{rgb}{0,.48,.65} 
\definecolor{dred}{rgb}{.5,0,0} 
\definecolor{dgreen}{rgb}{0,.3,0} 
\definecolor{vdred}{rgb}{.3,0,0} 
\definecolor{salmon}{rgb}{0.98,0.50,0.45} 
\definecolor{seagreen}{rgb}{0.13,0.70,0.67} 
\definecolor{chartreuse}{rgb}{0.40,0.80,0.00}
\definecolor{cornflower}{rgb}{0.39,0.58,0.93} 
\definecolor{gold}{rgb}{0.80,0.68,0.00}

\newtheorem{theorem}{Theorem}[section]

\newtheorem{proposition}[theorem]{Proposition}
\newtheorem{lemma}[theorem]{Lemma}
\newtheorem{corollary}[theorem] {Corollary}

\theoremstyle{remark}

\newtheorem{example}[theorem]{Example}

\theoremstyle{definition} 
\newtheorem{definition}[theorem]{Definition}

\usepackage{aeguill}                 
\usepackage{graphicx}
\usepackage{amsmath,amsfonts,amssymb,amsthm}
\usepackage{pstricks} 
\usepackage{epstopdf}
\usepackage{srcltx}
\usepackage[utf8]{inputenc} 
\usepackage{hyperref} 
\usepackage[english]{babel} 
\usepackage{comment}

\newcommand{\ov}[1]{{\overline{#1}}}

\def\G{\Gamma} 
\def\<{\langle}
\def\>{\rangle} 
\def\Z{\mathbb{Z}}

\def\R{\mathbb{R}}
\def\N{\mathbb{N}}
\def\g{\gamma}
\def\im{{\rm{im}\, }}
\def\CL{{\rm{CL}}}
\newcommand{\set}[1]{\left\{#1\right\}}

\def\e{\epsilon}
\def\ee{\varepsilon_d}

\def\A{{\mathbb{A}}}
\def\Ad{{\mathbb{A}}_d}

\def\ni{\noindent}

\begin{document}

\title[Conjugators in the model filiform groups]{The lengths of conjugators \\ in the model filiform groups} 

\author{M.\ R.\ Bridson and T.\ R.\ Riley}
 
\date{29 May 2025; revised 9 January 2026}

\begin{abstract}
\ni 
The conjugator length function of a finitely generated group $\G$ gives the optimal upper bound on
the length of a shortest conjugator for any pair  of conjugate elements in the ball of radius $n$ in
the Cayley graph of $\G$.  We prove that polynomials of arbitrary degree arise as conjugator length functions of finitely presented groups.  To establish this, we analyse 
the geometry of conjugation in  
the discrete model filiform groups $\G_d = \Z^d\rtimes_\phi\Z$ where 
$\phi$ is the automorphism of $\Z^d$ that fixes the last element of a basis $a_1,\dots,a_d$ and sends $a_i$ to $a_ia_{i+1}$ for $i<d$. The conjugator length function of $\G_d$ is polynomial of degree $d$. 
  \medskip
%\\
\ni \footnotesize{\textbf{2020 Mathematics Subject Classification:  20F65, 20F10, 20F18}}  \\[3pt] 
\ni \footnotesize{\emph{Key words and phrases:} nilpotent groups, model filiform groups, conjugator length}
\end{abstract}

\thanks{The first author thanks the Mathematics Department of Stanford University for its hospitality. The second author gratefully acknowledges the financial support of the National Science Foundation (NSF GCR-2428489). ORCID: 0000-0002-0080-9059 (MRB)  and 0009-0004-3699-0322 (TRR)}

\maketitle

\section{Introduction}
Whereas Dehn (isoperimetric) functions are the most natural measure of the complexity of a direct approach
to the word problem in a finitely presented group,  {\em conjugator length functions} are the most
obvious measure of the complexity of a direct approach to the conjugacy problem in a finitely
generated group.  By definition,
$\CL_G(n)$ is the least integer $N$ such that whenever  a pair of words $u$ and $v$ of length at most $n$
in the generators   represent conjugate elements of $G$, there is a word $w$ of length at most $N$ such that  $uw=wv$ in $G$.  
The study of conjugator length functions is far less developed than the study of Dehn functions. 
In particular, whereas it is essentially known which functions arise as Dehn functions \cite{BORS}, 
the class of functions that arise as  conjugator length functions is poorly understood. Indeed, until
recently, very few sharp estimates on conjugator length functions had been established. In part, this
reflects the greater delicacy of the conjugacy problem: for example, there exist pairs of finitely generated
groups $H<G$ with $|G/H|=2$ such that $H$ has a solvable conjugacy problem and $G$ does not \cite{CollinsMiller}.
Such examples warn us that, in contrast to the study of Dehn functions, we cannot hope to understand conjugacy length functions using the techniques of coarse geometry. On the other hand, the conjugacy problem is 
fundamentally 
geometric in nature:   understanding conjugacy in a group $\G$ is closely tied to an understanding 
of the geometry of annuli with prescribed boundaries in any space with fundamental group $\G$. We
refer to our article with Andrew Sale \cite{BrRiSa} for a more detailed introduction to these ideas and a survey
of what is known about conjugator length functions. 

In this article 
we will advance the state of the art by showing that polynomials of arbitrary degree arise as
conjugator length functions of finitely presented groups.  
We shall do so by investigating the 
geometry of the conjugacy problem in the discrete model filiform groups $\G_d$.
By definition, 
$\G_d = \Z^d\rtimes_{\phi_d}\Z$, where, with respect to a fixed basis $a_1,\dots,a_d$ for $\Z^d$, the automorphism
$\phi_d$ fixes $a_d$ and maps $a_i\mapsto a_ia_{i+1}$ for $i=1,\dots,d-1$.  
\begin{equation*}  
\G_d = \< a_1, \dots, a_d, t \mid \forall 1\le i<j\le d,\ [a_i,a_j]=1=[t,a_d],\  t^{-1}a_it = a_ia_{i+1}\>.
\end{equation*}
The corresponding nilpotent
Lie groups $G_d=\R^d\rtimes_\phi\R$ have been  studied extensively in the context of Carnot geometry, while
the lattices $\G_d<G_d$ have served as key examples for nilpotent and polynomial 
phenomena in geometric group theory.
In particular, $(\G_d)_{d\in\N}$ was among the first families of groups used to prove Dehn functions can be polynomial
of arbitrary degree; see \cite{BMS,  BridsonPittet}. 
Here, we prove that the groups $(\G_d)_{d\in\N}$ illustrate
the same diversity of behaviour among conjugator length functions.
In a companion article \cite{BrRi1}, we will show that 2-step nilpotent groups can also have conjugator length 
functions of  arbitrary polynomial degree.  A polynomial upper bound on the conjugator length functions of  
arbitrary finitely generated nilpotent groups was established much earlier \cite{MMNV}. %%new ref here

\begin{theorem}\label{t:main-poly}
The conjugator length function of $\G_d$ is polynomial of degree $d$.
\end{theorem}

Theorem~\ref{t:main-poly} provides a counterpoint to our work with Andrew Sale on the conjugator length functions of  free-by-cyclic groups: in  \cite{BrRiSa2} we proved that $\CL(n) \simeq n$ for
the  group $F_d\rtimes\Z$ obtained by removing 
the relations $[a_i,a_j]=1$ from the presentation of $\G_d$. It also contrasts with Sale's result \cite[Theorem~5.4]{Sale3} that if $\theta$ is   diagonalisable over $\R$ with positive eigenvalues,  then $\Z^d\rtimes_\theta \Z$ has $\CL(n) \simeq n$.

\bigskip

\ni{\bf{An outline of the proof.}} Whenever one tries to understand a conjugacy problem in a group $\G$, one is inevitably led to study the structure
of centralisers in $\G$, because the set of solutions $x$ to an equation $x^{-1}\g_1x=\g_2$ is a coset of 
the centraliser $C(\g_1)$. Thus a feature of our study is that we will need a close understanding of the structure of centralisers in 
$\G_d$; this is described in Section~\ref{s:cent}. It will emerge from this analysis of centralisers
 that we also need tight
control on the geometry of roots in $\G_d$, by which we mean solutions to equations $x^p=\gamma$; this is
the subject of Section~\ref{s:roots}. With these tools in hand, we will prove Theorem~\ref{t:main-poly} by
means of an induction on $d$ that exploits the observation that the centre of $\G_d$ is infinite cyclic, 
generated by $a_d$, and the quotient of $\G_d$  by its centre
is isomorphic to $\G_{d-1}$, with the images of $t$ and $a_i\ (i<d)$ satisfying the defining  relations of $\G_{d-1}$. When working with this inductive structure, it is important to keep in mind
that if two words $u,v$ in the letters $t,a_1,\dots,a_{d-1}$ define the same element in $\G_{d-1}$,
they will in general not be equal in $\G_d$; rather, an equality of the form $u=va_d^r$ will hold in $\G_d$.
The first lesson to be taken from this observation is that we must be careful to specify in which group equalities
between words are taking place. The second point to absorb is that we will have reason to control $|r|$ in 
expressions of the form $u=va_d^r$.  

With these preparatory thoughts in mind, we can outline our strategy of proof. The actual proof requires
us to keep track of various constants, but for the purposes of this outline it makes sense to absorb these
into symbols such as $\preceq$.

The lower bound $n^d\preceq \CL_{\G_d}(n)$ is established (Section~\ref{s:normal forms}) by arguing that $t^r$ is the unique shortest
word conjugating $a_{d-1}$ to $a_{d-1}a_d^r$ in $\G_d$ and that $d(1,\, a_{d-1}a_d^{n^d}) \simeq n$.  En route we explain how the Hilbert-Waring Theorem facilitates a short combinatorial treatment of the metric geometry of $\G_d$.

The argument to establish that $\CL_{\G_d}(n)\preceq n^d$ is more involved. 
We must show that if $u$ and $v$ are conjugate in $\G_d$ and
both lie in the ball of radius $n$ about $1$, then there exists $x$ such that  $x^{-1}ux=v$  and $|x|\preceq n^d$.  
We will see that the difficult case to deal with is when $v$ does not lie in the normal
subgroup $\Z^d\rtimes 1$, so let us assume here that this is the case.
Assuming that the theorem has been proved for $\G_{d-1}$, we get a word $\theta$ with length $|\theta|\preceq n^{d-1}$
in the letters $t, a_1,\dots, a_{d-1}$ conjugating the images of $u$ and $v$ in $\G_{d-1}$. 
This tells us that  $\theta^{-1}u\theta=va_d^r$ in $\G_d$ for some integer $r$.
 We will use our knowledge of the Dehn function in $\G_{d-1}$ to ensure that $|r|\preceq n^{d(d-1)}$.
 To complete the induction, we must  
 argue that if $va_d^r$ is conjugate to $v$ and $|r| \preceq n^{d(d-1)}$, then the conjugacy can be achieved
 by a conjugator $h$ with $d(1,h) \preceq n^d$. (Note that $va_d^r$ may lie outside the ball of radius $n$
 in $\G_d$, which complicates the structure of the argument.)
 
 It transpires that it is best to realise the conjugacy of $v$ to $va_d^r$ in two steps. In the first step,
we conjugate by a suitable power of 
 $a_{d-1}$: the aim here is to reduce $|r|$
 to something of significantly smaller order by conjugating by $a_{d-1}^M$ with $M\approx n^{d(d-1)}$, and
 since the distortion of $\<a_{d-1}\><\G_d$ is polynomial of degree $d-1$,  the length of such a conjugator is
roughly $n^d$, which is what we are aiming for. We will see that, under the constraint 
$M\approx n^{d(d-1)}$, we
can reduce $|r|$ to something less than $pqe<ne$ where $p$ is the
maximal power such that the image of $v$ in $\G_{d-1}$ is an element that can be written in the form $\ov{v}=v_0^pa_{d-1}^*$
and $q$ is the exponent sum of $t$ in $v_0$,  and $e = \gcd(p,q)$. A surprisingly large amount 
of work is required for the final step in the proof, wherein we must bound the length of a conjugator
taking $v$ to $va_d^r$ for small values of $r$: 
this involves understanding the action on the coset $v\<a_d\>$ in $\G_d$
of the centraliser of $\ov{v}\in\G_{d-1}$   through  a
homomorphism $\zeta_{\ov{v}}: C(\ov{v})\to \Z$ defined in Definition~\ref{d:def}.  Our understanding 
of this action relies on  an estimate  $d(1,v_0)\preceq d(1,v)$ that comes from the
analysis in Section~\ref{s:roots}  of the lengths of roots in $\G_{d-1}$.

\subsection*{Acknowledgement}   We thank the referees for their careful reading and helpful comments.

\section{Preliminaries}

Throughout,  $\G_d$ will denote the group $\Z^d\rtimes_{\phi_d}\Z$ where, with respect to
a fixed basis, $\phi_d$ and $\phi^{-1}_d$ are represented by the following matrices, respectively:
$$  \left( 
 \begin{array}{cccccc}
  1 &  &  &  &  \\
  1 & 1 &  &  &  \\
   & 1 & \ddots &  &  \\
   &  &   \ddots & 1  &  \\
   &  &   & 1 & 1 \\
\end{array}% 
\right) \quad \text{ and  } \quad     \left(%
 \begin{array}{cccccc}
  1 &  &  &  &  \\
  -1 & 1 &  &  &  \\
   1  & -1 &  \ddots &  &  \\
   -1  & 1 &  \ddots & 1  &  \\
   \vdots &   \vdots &  \ddots & -1 & 1 \\
\end{array} 
\right).$$ When there is no danger of ambiguity, we write $\phi$ instead of $\phi_d$.
We work  with  the standard presentation
\begin{equation}\label{Gd}
\G_d = \< a_1,\dots,a_d,t\mid \forall 1\le i<j\le d,\ [a_i,a_j]=1=[t,a_d],\  t^{-1}a_it = a_ia_{i+1}\>.
\end{equation}
All distances in $\G_d$ will be taken with respect to  the generating set of this presentation and the area 
of any null-homotopic word will be calculated with respect to this set of defining relations. It is worth noting,
however, that $\G_d$ is generated by just two elements, $a_1$ and $t$.

We are interested in the  \emph{conjugator length function} $\CL : \N \to \N$ of these groups.
By definition, $\CL(n)$ is the minimal integer $N$ such that whenever  a pair of words  $u$ and $v$ 
of length at most $n$ in the generators
represent conjugate elements in $\G_d$, 
there is a word $w$ of length at most $N$ such that  $uw=wv$ in $\G_d$.     
If we worked with a different finite generating set, the
resulting conjugator length function would be equivalent
 in the sense of the following standard definition.  
 \smallskip
 
 \ni{\em For functions $f$ and $g$ mapping $\mathbb{N} \to \mathbb{N}$ write $f \preceq g$ when there exists $C>0$ such that   $f(n) \leq Cg(  Cn+C ) +C$  for all $n \in \N$, and write $f  \simeq g$ when $f \preceq g$ and $g \preceq f$. }

The following standard lemma will be needed   to control 
the size of the discrepancy $r$ that arises when we lift an equality $u=v$ in   $\G_{d-1}$ to
an equality of the form $u=va_d^r$  in $\G_d$. 

\begin{lemma}\label{l:dehn} If two words $u,v$ in the letters $t,a_1,\dots,a_{d-1}$ define the same element of $\G_{d-1}$, then $u=va_d^r$ in $\G_d$, where  $|r|\le {\rm{Area}}_{\G_{d-1}}(uv^{-1})$.
\end{lemma}

\begin{proof}
By definition,  ${\rm{Area}}_{\G_{d-1}}(uv^{-1})$ is the least integer $A$ for which there is an equality 
$$
uv^{-1} = \prod_{i=1}^A \theta_i^{-1} \rho_i \theta_i
$$
in the free group $F=F(t,a_1,\dots,a_{{d-1}})$
with each $\rho_i^{\pm 1}$ one of the defining relations of $\G_{d-1}$. We modify the righthand side of this
expression by replacing each instance of $[a_{d-1}, t]$ with $[a_{d-1}, t]a_d$, leaving the other relations
untouched. Let $\rho_i'$ be the resulting relations. Then, in $F\times \<a_d\>$ we have the equality
$$
uv^{-1} =  a_d^r\prod_{i=1}^A \theta_i ^{-1}\rho_i' \theta_i,
$$
where $-r$ is the number of letters $a_d$ added during the editing of the $\rho_i$. The  righthand side has image
$a_d^r$
 under the surjection $F\times\< a_d\>\to \G_d$ that is implicit in the notation. Thus $u=va_d^r$ in $\G_d$, with
$|r|\le A$.
\end{proof}

The following well-known result 
will be needed in connection with the above lemma.

\begin{theorem}[\cite{BridsonPittet}]\label{t:dehn}
The Dehn function of $\G_d$ is $\simeq n^{d+1}$.
\end{theorem}

Positive powers of the matrices for $\phi$ and $\phi^{-1}$ can be calculated using the binomial theorem,
writing each as $I+N$ and expanding $(I+N)^n$.
For $\phi$, the entry in column $i$ on the bottom row for $\phi^n$ 
is $\e(i,n)=\binom{n}{d-i}$, while the corresponding entry $\e(i,-n)$ for $\phi^{-n}$ is a signed sum on
binomial coefficients.  
The following crude estimate will suffice for our purposes.

\begin{lemma}\label{l:entry}
There is a constant $\ee$ such that $|\e(i,m)| \le \ee\, |m|^{d-i}$ for all $m\in\Z$ and $i=1,\dots,d-1$.
\end{lemma}

The significance for us of the  integers $\e(i,n)$ is that for $i=1,\dots,d-1$, 
if $\phi_{d-1}^n(a_i)\in \G_{d-1}$ and $\phi_{d}^n(a_i)\in \G_{d}$ are written as words in 
the letters $a_1,\dots,a_d$, then in $\G_d$ we have
\begin{equation}\label{e:error}
\phi_{d}^n(a_i) = \phi_{d-1}^n(a_i) \, a_d^{\e(i,n)}.
\end{equation}

\section{Normal forms and distance} \label{s:normal forms}

The essential result for us concerning navigating $\G_d$ metrically is:  

\begin{theorem}\label{t:dist-cyclics}
For $i=1,\dots,d$, the function $n\mapsto d_{\G_d}(1,a_i^n)$ is Lipschitz equivalent to $n^{1/i}$.
\end{theorem}

%\begin{remark}\label{r:carnot}
This theorem played an important role in \cite{Bridson3} where it was proved with an argument from \cite{Gui, gromov}. This argument takes place  in the Malcev completion of $\G_d$, which is the nilpotent Lie group $G_d=\R^d\rtimes\R$ where $\G_d$ is a lattice. The Lie algebra of $G_d$, which is graded, has presentation
$$
( X_1,..., X_d, T : [X_i,X_j]=0=[X_d,T],\ [T,X_i]=X_{i-1},\  i=1,...,d-1 ).
$$
This Lie algebra is generated under brackets by $\{X_1, T\}$, so there is a Carnot-Caratheodory metric on $G_d$: this  
is defined by fixing an inner product on the plane $\<X_1, T\>$ so that $\{X_1, T\}$ is orthonormal, transporting this plane around $G_d$ by left translation, then defining the distance $d_{cc}(x,y)$ between two points $x,y\in G_d$ to be the infimum of the lengths of piecewise
smooth curves from $x$ to $y$ that are almost everywhere tangent to this field of 2-planes, with the length of tangent vectors measured using the transported inner product. 

The Lie algebra of $G_d$ admits the 1-parameter family of automorphisms $m_s: X_i \mapsto s^i.X_i,\ T\mapsto s.T$ for $s\neq 0$.
These automorphisms induce (via the exponential map) homotheties of the Carnot-Caratheodory metric 
on $G_d$.  For all $s\in\mathbb{R}$ we have $d_{cc}(1, \exp(sX_i)) =  s^{1/i}$. The restriction of $d_{cc}$
to the lattice $\G_d$ is Lipschitz equivalent to the word metric on $\G_d$ and $a_i^n= \exp(nX_i)$.
%\end{remark}

\bigskip

Here we give an alternative, more combinatorial proof of Theorem 3.1. This proof invokes 
Hilbert's solution to Waring's problem, but is otherwise elementary.  Lemma~\ref{l:normal-form} and Proposition~\ref{p:whole-ball} below will establish the complementary upper and lower bounds for the theorem.

We will push the letters $t^{\pm 1}$ to the left when writing elements
of $\G_d$ in the standard normal form for polycyclic groups.

\begin{lemma}\label{l:normal-form}
Every element $g\in \G_d$ can be expressed uniquely as a word in normal form $t^r a_1^{p_1}a_2^{p_2}\cdots a_d^{p_d}$,
and if $d(1,g)\le n$ then $|r|\le n$ and $|p_i|\le n^i$ for $i=1, \ldots, d$. 
\end{lemma}

\begin{proof} The only non-trivial assertion is the one about the size of the exponents. We proceed
by induction on $n$. 
Given a word $w=b_1\cdots b_n\beta$ of length $n+1$, we apply the inductive hypothesis to write
$w'=b_1\cdots b_n$ in normal form, then shuffle $\beta$ past the letters of this normal form to write $w$ in normal
form. If $\beta=a_j^{\pm 1}$  then this shuffling leaves $r$ and $p_i$ unchanged for $i\neq j$ and alters $p_j$
by $\pm 1$. If $\beta=t$ then this shuffling leaves $p_1$ unchanged
and for $i\ge 2$ increases  $p_i$ by $p_{i-1}$, so the  observation 
$n^i+n^{i-1} < (n+1)^i = n^i + in^{i-1}+\cdots $ finishes the induction. 
If $\beta=t^{-1}$
then the shuffling replaces $p_i$ by  
$$
(-1)^{i+1} (p_1 - p_2 +\dots + (-1)^{i+1}p_{i}),
$$
whose absolute value is, by induction, less than  $n + n^2 +\dots + n^{i}$,
which is less than $(n+1)^i$.
\end{proof}

Our next result provides a converse to the preceding lemma. 

\begin{proposition} \label{p:whole-ball}
There is a constant $C_d {\geq 1}$ such that for every $g\in\G_d$ with normal form 
$t^ra_1^{p_1}a_2^{p_2}\dots a_d^{p_d}$, where
  $|r|\le n$ and $|p_i|\le n^i$, we have  $d(1,g)\le C_d\, n$.
\end{proposition}

The following special case of the proposition contains the main idea in the proof; it also
facilitates a simple induction to prove the general case. 

\begin{lemma} \label{l:real}
For every $d\ge 1$ there is a constant $D_d$ such that for  all positive integers $0<p\le n^d$
we have $d(1,a_d^p) \le D_d\, n$ in $\G_d$.
\end{lemma}

\begin{proof}
We proceed by induction on $d$. 
The base case is $\G_1=\Z^2$, where the assertion is trivial.

The Hilbert-Waring Theorem provides us with an integer $M_d$ such that every positive integer $p$ can be written as
a sum $p=k_1^d + \dots + k_m^d$ with $k_i\in\N$ and $m\le M_d$. 
   Then $\sum_{i=1}^m k_i \leq M_d n$, because $p \leq n^d$ implies that $k_i \leq n$ for all $i$. 
  In $\G_d$, for any $k\in\N$ we have the equation   
$${a_{d-1}^{-{k^{d-1}}} t^{-k} a_{d-1}^{k^{d-1}} t^k = a_{d}^{k^{d}}.}$$ 
By induction, for $i=1, \ldots, m$, there is a word $w_i$  in the letters $t,a_1,\dots,a_{d-1}$ of length at most $D_{d-1}\ k_i$
such that $w_i = a_{d-1}^{k_i^{{d-1}}}$ 
 in $\G_{d-1}$. Recalling that $\G_{d-1}$ is obtained from $\G_d$ by killing
the central subgroup $\<a_d\>$, we have $w_i = a_{d-1}^{k_i^{{d-1}}}a_d^{l_i}$ in $\G_{d}$ for some $l_i\in\Z$.  
As $a_d$ is central, in $\G_d$ we have 
$${w_i^{-1} t^{-k_i} w_i t^{k_i} = a_{d-1}^{-k_i^{d-1}} t^{-k_i} a_{d-1}^{k_i^{d-1}} t^{k_i} = a_d^{k_i^d}}.$$
Hence, in $\G_d$,   
$$
{a_d^p = \prod_{i=1}^m w_i^{-1} t^{-k_i} w_i t^{k_i},}
$$
and the righthand side is a word of length at most  
$2(D_{d-1} +1)\sum_{i=1}^m k_i  \le 2(D_{d-1} +1)M_d \, n$. 
Taking
${D_d = 2(D_{d-1} +1)M_d}$  
completes the induction.
\end{proof}

\begin{proof}[Proof of Proposition \ref{p:whole-ball}.] 
We proceed by
induction on $d$. Again, the base case $d=1$ is trivial. Assume that the proposition is true in $\G_{d-1}$.
Given  $g  = t^r a_1^{p_1}a_2^{p_2}\dots a_d^{p_d}\in\G_d$,
with $|r|\le n$ and each $|p_i|\le n^i$, we know by induction that there is a word $W=W(t,a_1,\dots,a_{d-1})$ 
of length at most $C_{d-1}n$ such that in $\G_{d-1}$ we have
$${W=t^r a_1^{p_1}a_2^{p_2}\cdots a_{d-1}^{p_{d-1}}.}$$  
In $\G_d$, then, for  a unique $q\in\Z$ we have
$$
W= t^r a_1^{p_1}a_2^{p_2}\cdots a_{d-1}^{p_{d-1}} a_d^{q},
$$ 
and by Lemma \ref{l:normal-form} we know $|q|\le  |W|^{d} \le C_{d-1}^d n^d$. 
So, $|p_d -q| \le |p_d| + |q| \leq (1+C_{d-1}^d) n^d$, and then Lemma~\ref{l:real} leads to $d(1, a_d^{p_d-q}) \leq (1 + C_{d-1}^d) D_d n$.
Therefore, by the
triangle inequality in $\G_d$, noting that $g = Wa_d^{p_d-q}$, $$
d(1,  g) \le d(1,W) + d(1, a_d^{p_d-q}) \le C_{d-1}n + D_d (1+C_{d-1}^d) n,
$$
and defining $C_d = C_{d-1}+D_d(1 + C_{d-1}^d)$  completes the induction.   
\end{proof}

\subsection*{Lower Bound for Theorem \ref{t:main-poly}}

 The key observation here is
 that an arbitrary element  $t^rx\in\G_d  {=  \Z^d\rtimes_\phi\Z}$ with $x\in \Z^d\rtimes 1$ conjugates $a_{d-1}$ to $a_{d-1}a_d^r$,
 with $x$ playing no role. It follows that for each positive integer $r>0$,
 the unique shortest element of $\G_d$ conjugating $a_{d-1}$ to  $a_{d-1} a_d^r$ is  $t^r$. 
 
 In Lemma \ref{l:real} we established the existence of a constant $D_d$ such that $a_d^{n^d}$ lies
 in the ball of radius $D_d n$ about $1\in \G_d$, hence $d(1, a_{d-1} a_d^{n^d})\le D_d n + 1$.
 As $t^{n^d}$ is the unique shortest element conjugating $a_{d-1}$ to $a_{d-1}a_d^{n^d}$, we deduce
 $$
 \CL_{\G_d}(D_d n + 1) \ge n^d,
 $$ 
 hence $n^d\preceq \CL_{\G_d}(n)$.

\section{Centralisers in $\G_d$} \label{s:cent}

Picking up a theme from the last argument of the previous section, we note that $g\in\G_d$ will 
conjugate $\g$ to an element of the form   $\g a_d^*$ if and only if the image of $g$ in $\G_{d-1}$
centralises the image of $\g$. This highlights our need to understand centralisers $C(\gamma)$ in the groups $\G_d$.

\smallskip
\ni{\bf{Notation.}}
We will need to discuss the subgroup $\<a_1,\dots,a_d\> = \Z^d\rtimes 1 < \G_d$ repeatedly in what follows, and it
will be important to distinguish it from an abstract copy of $\Z^d$. For this reason, we shall henceforth 
denote it by $\Ad$.

\begin{proposition}\label{p:centr}
Given $\gamma\in\G_d$, if $\gamma = \gamma_0^p a_d^q$ with $p>0$ maximal, then 
\begin{enumerate}
\item $C(\gamma) = \<\gamma_0, a_d\>\cong\Z^2$ if $\gamma\not\in \Ad$,
\item $C(\gamma) = \Ad$ if $\gamma\in\Ad\smallsetminus \<a_d\>$,
\item $C(\gamma)=\G_d$ if $\gamma\in\<a_d\>$.
\end{enumerate}
\end{proposition}

\begin{proof} We again proceed by induction on $d$, viewing $\G_{d-1}$ as the quotient of $\G_d$
by $\<a_d\>$ and writing $\overline{\g}$ for the image of $\gamma$ in $\G_{d-1}$. The base
step $\G_1\cong\Z^2$ is trivial (with case (2) vacuous). 

Cases (2) and (3) are easily verified, so in the inductive step we 
concentrate on the case $\gamma= \gamma_0^p a_d^q$
with $\gamma_0\not\in \Ad$.  

Some preliminary observations are in order: if $\ov{\g} = g^m$ in $\G_{d-1}$, then for any preimage $\tilde{g}$ 
in $\G_d$ of $g$ we have $ \tilde{g}^m = \g a_d^r$ for some $r$, so $m\le p$, by the maximality of $p$, and if $m=p$ then
$g = \ov{\gamma_0}$ by the uniqueness of roots in torsion-free nilpotent
groups. Thus $\ov{\gamma_0}$ is the unique maximal root of $\ov{\gamma}$ in $\G_{d-1}$
(even though $\gamma_0$ itself is only uniquely defined up to multiplication by a power of $a_d$). 

By induction, $C(\ov{\gamma})<\G_{d-1} $ is a free abelian {group} of rank $2$ that contains $\<a_{d-1}\>$. The image of 
 $C(\gamma)$ in $C(\ov{\gamma})$   intersects $\<a_{d-1}\>$ trivially, because  
 either $\g$ or $\g^{-1}$ is
$xt^{-j}$ with $x\in\Ad$ and $j> 0$, and  in $\G_{d}$ (n.b. $\G_d$ rather than $\G_{d-1}$)
$$ \tilde{a}_{d-1}^{-k}xt^{-j} \tilde{a}_{d-1}^{k} =  a_d^{kj} xt^{-j} \neq xt^{-j}$$
for any preimage $\tilde{a}_{d-1} = a_{d-1} a_d^*$ in $\G_d$ of $a_{d-1}\in\G_{d-1}$.
Thus the image of $C(\gamma)$ in $C(\ov{\gamma})$  is a cyclic group. 
This cyclic group contains $\ov{\gamma}$,
so by the considerations of the previous paragraph, it is contained in 
$\<\ov{\g_0}\>$. As $\g_0$ commutes with $\g$, the image is exactly $\<\ov{\g_0}\>$.

The preimage in $\G_d$ of $\<\ov{\gamma_0}\>$ is $\<\gamma_0, a_d\>$, which centralises $\gamma$. Thus
$C(\gamma) = \<\gamma_0, a_d\>$ and the induction is complete.
\end{proof}

\section{The lengths of roots in $\G_d$} \label{s:roots}

The emergence of $\gamma_0$ in Proposition \ref{p:centr}
%our discussion of centralisers in the previous section %MB
 throws up a need to understand roots in the groups $\G_d$.
In particular, when it comes to estimating the lengths of minimal conjugating elements, we will need to bound the
lengths of roots of elements in $\G_d$.

The bounds in the following proposition are not optimal, but
they will suffice for our purposes and a sharpening would involve an unwelcome increase in
notation.

\begin{proposition}\label{p:roots} There is a constant $k_d>0$ so that for all $p, n>0$ and all
 $\g\in\G_d$ with $d(1,\g)\le n$, %has normal form $t^{e_0}a_1^{e_1}\dots a_d^{e_d}$
if $h^p = \g$ then  
$$h = t^{m_0}a_1^{m_1}\cdots a_d^{m_d}$$ 
with $|m_0|\le n/p$ and
$|m_i| \le k_d n^i/p$ for $i=1,\dots,d$.
\end{proposition}

For the most part, we shall only need the following easy consequence of Propositions \ref{p:whole-ball} 
and \ref{p:roots}.
 
 \begin{corollary}\label{c:roots} There is a constant $K_d$ so that
 if $h^p=\g$ in $\G_d$ and $d(1,\g)\le n$, then $d(1,h)\le K_d\, n.$
 \end{corollary}
 
Before proving Proposition \ref{p:roots}, we consider the case $d=2$, which admits a more direct argument.

\begin{example}
In the 3-dimensional Heisenberg group $\G_2 = \<a_1,a_2,t\>$ we have
$$
(a_1t)^n = t^na_1^n a_2^{\sigma(n)},
$$
where 
$$\sigma(n) = 1+\dots + n = n(n+1)/2={\rm{Area}}_{\<a_1,t\>}(u^nv^{-1})$$
 with $u\equiv a_1t$ and $v\equiv t^na_1^n$.  
One can see that $\sigma(n)$ really is ${\rm{Area}}_{\<a_1,t\>}(u^nv^{-1})$ either by considering the van Kampen 
enclosed by the loop $u^nv^{-1}$ in the $(a_1,t)$-plane tessellated by squares, or else by arguing algebraically:
we reduce $u^nv^{-1} \equiv (a_1 t)^n a_1^{-n} t^{-n}$ to the empty word by 
starting in the middle and replacing the subword $a_1 t a_1^{-1}$ by $t$ at a cost of applying $1$ commutator relation, then $a_1 t^2 a_1^{-1}$ by $t^2$ at an additional cost of $2$, etc., until  finally  $a_1 t^n a_1^{-1}$ is replaced by $t^n$ at a cost of $n$, after which we can freely reduce $t^n t^{-n}$ to the empty word.

If $n$ is odd, then
 $$
(a_1ta_2^{-(n+1)/2})^n = t^na_1^n.
$$
Thus $h_n:=a_1ta_2^{-(n+1)/2}$ is the unique $n$-th root of $g:=t^na_1^n$ in $\G_2$. Projecting to $\G_1=\Z^2$,
it is easy to see that $d(1,t^na_1^n)=2n$ in $\G_2$. On the other hand,   $d(1,h_n)$ is approximately 
$\sqrt{n}$, even though $h_n^n=g$.

We now consider the more general situation where 
$\gamma\in\G_2\smallsetminus\<a_1,a_2\>$ is a $p$-th power, say $\g=h^p$. Counting occurrences of $t$,
we see that $|p|\le d(1,\g)$. Our goal is to estimate the length of $h$.
 Suppose $n=d(1,\g)$ and fix a geodesic word $w$ representing $\g$. Let $\overline{w}$ be the word obtained
 from $w$
by deleting all occurrences of $a_2$ and note that this word equals the image of $\g$ in $\G_1$. We
write the image of $h$ in $\G_1=\Z^2$ as a geodesic word $v$ in the free {abelian}   group on $\{a_1,t\}$.
Replacing $\g$ and $h$ by their inverses if necessary, we may assume that only positive powers of $t$ arise in $v$.
 In $\G_1=\Z^2$ we know that $v^p=\ov{w}$ implies $|v|= d(1, \ov{w})/p$, so $|v|\le |w|/p = n/p$.
As in Lemma \ref{l:dehn}, in $\G_2$ we have 
$$
v^p= w a_2^r,
$$
where $|r|\le {\rm{Area}}_{\G_1}(v^{-p} \ov{w})$,  
which is less than $n^2$, by  
the quadratic isoperimetric inequality in $\Z^2$.  (In more detail, 
$v^{-p}\ov{w}$ is a non-empty word $\tau$ in $a_1$ and $t$ of length at most $2n$ that equals the identity in $\Gamma_1$
and any such $\tau$ has a subword which freely reduces to $a_1^{\pm 1} t^l a_1^{\mp 1}$ or $t^{\pm 1} a_1^l t^{\mp 1}$ for some $|l|\leq n/2$;  this subword can be replaced by $t^l$ or $a_1^l$, respectively, at a cost of applying at most $n/2$ commutator relations, from which it follows that ${\rm{Area}}(\tau) < n^2$, by an induction on length.)  
But $v=ha_2^j$ in $\G_2$ for some $j\in\Z$, so 
$$wa_2^r = v^p = h^pa_2^{jp} = \gamma a_2^{jp} = w a_2^{jp}.$$ 
Therefore $r=jp$.

We saw earlier that $|v|\le n/p$ and $|r| {<}  n^2$, 
and from Proposition \ref{p:whole-ball} we know $d(1, a_2^{-j}) \le C_2\, \sqrt{|j|} = C_2\sqrt{|r/p|}$, which is less than $C_2\, n/\sqrt{p}$. Thus,  
taking $v\in\<a_1, t\>$ to be $v=t^{m_0}a_1^{m_1}$ with $|m_0| + |m_1| \le n/p$,
we have $h = va_2^{-j} = t^{m_0}a_1^{m_1}a_2^{{-j}}$ 
in the form required by Proposition \ref{p:roots},
since $|j| \le n^2/p$. Moreover,  
$$
d(1,h)\le |v| + d({1}, a_2^{{-j}})  \le  \frac{1}{p} (n  + C_2\, np^{1/2}) \le C_2\ \frac{1}{p} (n  + np^{1/2})\le 2C_2 \ np^{-1/2}.
$$ 
\end{example}

\begin{proof}[Proof of Proposition \ref{p:roots}.] We proceed by induction on $d$. The case $d=1$ is trivial and the case $d=2$
was covered in the preceding example. Assume true up to $d-1$. We have $\gamma, h \in \Gamma_d$ with $h^p = \gamma$ and $d(1, \gamma) \leq n$.  Let $\ov{h}$ be the image of $h$ in $\G_{d-1}$. Then, by induction, $\ov{h} = t^{m_0}a_1^{m_1}\dots a_{d-1}^{m_{d-1}}$,
in $\G_{d-1}$, with the $|m_i|$ bounded as stated in the proposition. Let 
$$h_0:=t^{m_0}a_1^{m_1}\cdots a_{d-1}^{m_{d-1}}\in \G_d.$$
Then $h_0 = ha_d^m$ for some $m\in\Z$, and
if the normal form for $h_0^p\in\G_d$ is
$$
h_0^p = t^{pm_0}a_1^{s_1}\cdots a_d^{s_d},
$$
then the normal form for $\gamma   = h^p = (h_0 a_d^{-m})^p$   is
$$
\gamma = t^{pm_0} a_1^{s_1}\cdots a_{d-1}^{s_{d-1}} a_d^{s_d-mp},
$$ 
and from Lemma~\ref{l:normal-form} we have $|s_d - mp|\le n^d$. So, if
we can prove that there is a constant $\alpha_d$ such that 
\begin{equation}\label{e:alpha}
 |s_d| \le \alpha_d n^d,	
\end{equation}
then we will have 
$|mp| \le (\alpha_d+1)n^d$
and defining $k_d:= \max \set{ \alpha_d +1, k_{d-1} }$   will complete the induction.
 
To obtain the required bound on $s_d$, we transform $h_0^p$ into normal form carefully, pushing all occurrences of 
$t^{\pm 1}$ to the left and then shuffling the resulting element of $\Ad\cong\Z^d$ into normal form.
The pushing of the $t$ letters amounts to using free identities $a_it^{{j} m_0} = t^{{j}m_0}(t^{-{j}m_0}a_it^{{j}m_0})$ 
and then evaluating $t^{-jm_0}a_it^{jm_0}$ as $\phi_d^{jm_0}(a_i)$. At this point, it is 
important to recall from equation (\ref{e:error}) that  for each $i\le d-1$
$$
\phi_d^{jm_0}(a_i) = \phi_{d-1}^{jm_0}(a_i) a_d^{\e(i,jm_0)}.
$$
(This is equally 
valid for the negative powers of $\phi$ arising from pushing $t^{-1}$ if $m_0<0$.)
 
The result of shuffling $h_0^p$ into normal form in $\G_{d-1}$ (producing the word 
$t^{pm_0}a_1^{s_1}\cdots a_{d-1}^{s_{d-1}}$)
 differs from the result of doing it in $\G_d$ by 
a power of $a_d$ that can be calculated by considering how many copies of each letter $a_i$ in $h_0^p$
we have to move each $t^{m_0}$ past, adding up the number of $a_d$ letters produced by each push. The total 
(ignoring cancellation due to signs if $m_0<0$) is
$$
M = \sum_{j=1}^{p-1}  \sum_{i=1}^{d-1} |m_i\, \e(i,jm_0)|.
$$
We estimate  this (crudely) by using the inequality $|\e(i,m)| \le \ee\, m^{d-i}$ from 
 Lemma~\ref{l:entry}, noting that $|j m_0|$ is less than  
$|pm_0|$, which  is the absolute value of the exponent sum of $t$ in $\g$, which is less than $n$.
Thus 
$$M  \le  \sum_{j=1}^{p-1}  \sum_{i=1}^{d-1} |m_i\, \ee\, (jm_0)^{d-i}|
 \le
\ee\, (p-1) \sum_{i=1}^{d-1} |m_i| \, n^{d-i},  $$
and our induction assures us that $|pm_i| \le k_{d-1} n^i$, so  
$$
|s_d|\le M\le 
\ee\, k_{d-1} \sum_{i=1}^{d-1}   n^i n^{d-i} =   \ee\, k_{d-1} \sum_{i=1}^{d-1}  n^d.$$
Thus, to obtain the required bound (\ref{e:alpha}),  it suffices to set $\alpha_d := d\ee k_{d-1}$.  
\end{proof}

\section{Quantifying B\'{e}zout's Lemma} \label{s:bezout}

We will need an elementary observation concerning  B\'{e}zout's Lemma.

\begin{lemma}\label{l:euclid}
For positive integers $A$ and $B$,  if $e = \gcd (A,B)$
then there exist a non-negative integer $\mu<A$ and a positive integer $\lambda\le B$ such that  $\lambda A - \mu B =e$.
If $A>1$,  one can assume $\mu $ is positive.
\end{lemma}

\begin{proof} By B\'{e}zout's Lemma, there exist non-zero integers $a$ and $b$ such that $aA+bB=e$.
 If $A=1$,  take $\lambda=1$ and $\mu=0$.  If $A>1$ and $B=kA$,  
take $\lambda=k+1$ and $\mu =1$.  In all other cases,    
 $A$ does not divide $e$,  so it does not divide $b$,  and $b=kA-r$ with $0< r < A$. Then $(kB +a) A  - rB=e$ and 
$$
(kB+a) =  (r/A)B + (e/A)
$$
is a positive integer less than or equal to  $B$, because  $r/A$ and $e/A$ are both positive numbers less than $1$.
Let $\lambda = (kB+a)$ and let $\mu = r$.
\end{proof}

The simple estimate in the following corollary can be sharpened (cf.~\cite[Corollary 3.3]{BrRi3}) but it is sufficient for our needs in this article.

\begin{corollary}\label{c:new} Let $A, B$ and $m$ be integers and let $N=\max\{|A|,\, |B|\}$. If there is
an integer solution $(x,y)$ to the equation $Ax+By=m$, then there is an integer solution $(x_0,y_0)$
with $\max \{ |x_0|,\ |y_0|\} \le |m| + N^2$.
\end{corollary}

\begin{proof}
If $N= 0$ the assertion is trivial,  so suppose $|A|\ge 1$. 
Write $m=k A + r$ with  $|k|\le |m|$ and $0\le r <|A|$.  As $m$ lies in the subgroup of $\Z$ generated
by $A$ and $B$, so does $r$. This subgroup is generated by $e =  \gcd (A,B)$,  so $r=r_0e$
for some $r_0\le r<|A|$.  The lemma allows us to write $e= \lambda_0 A + \mu_0 B$ with 
$\max\{|\lambda_0|,\, |\mu_0|\} \le N$. Then
$$m = (k+\lambda_0r_0) A + \mu_0 B,$$
and $|k+\lambda_0r_0| \le |k| + r|\lambda_0|\le |m| + N^2$.
\end{proof}

\section{The proof of Theorem \ref{t:main-poly}}

At the end of Section \ref{s:normal forms}  we proved that $\CL_{\G_d}(n) \succeq n^d$, so the following proposition completes the
proof of Theorem~\ref{t:main-poly}.

\begin{proposition}\label{p:upper}
There is an integer $B_d$ such that if $u,v\in \G_d$ are conjugate and $\max\{ d(1,u),\, d(1,v)\} \le n$,
then there exists $g\in\G_d$ with $g^{-1}ug=v$ and $d(1,g)\le B_dn^d$. 
\end{proposition} 

We will again proceed by induction on $d$.
If $d=1$ then $\G_d=\Z^2$ and the proposition is trivial.  
{If $d=2$ then $\G_d$ is the 3-dimensional integral Heisenberg group and since $u$ and $v$ are conjugate in $\G_2$, their images in  $\G_{1} \cong \Z^2$ are equal.  So $u = t^r a_1^{\alpha} a_2^{\eta}$ and $v = t^r a_1^{\alpha} a_2^{\xi}$ for some $r, \alpha, \eta, \xi \in \Z$, and there exist $s, \omega \in \Z$ such that $g = t^s a_1^{\omega}$ satisfies $u g = g v$ in $\G_d$. 
Now  in $\G_2$, 
$$ug =   ( t^r a_1^{\alpha} a_2^{\eta})(   t^s a_1^{\omega})  =  t^{r+s} \phi^s( a_1^{\alpha} a_2^{\eta}) a_1^{\omega} = t^{r+s} a_1^{\alpha + \omega} a_2^{s\alpha + \eta}
$$
and 
$$gv =   (   t^s a_1^{\omega}) ( t^r a_1^{\alpha} a_2^{\xi})  =  t^{r+s} \phi^r( a_1^{\omega}) a_1^{\alpha} a_2^{\xi} = t^{r+s} a_1^{\alpha + \omega} a_2^{r\omega + \xi},
$$
 and these are equal if and only if $s\alpha + \eta = r\omega + \xi$.  This condition can be expressed as $s\alpha -   r\omega =  \xi - \eta$, and if we regard this as an equation with variables $s$ and $\omega$, 
 Corollary~\ref{c:new} implies that this equation has an integer solution with $|s|$ and $|\omega|$ both at most $|\xi - \eta| + \max \set{|\alpha|^2, |r|^2}$.  Lemma~\ref{l:normal-form}  tells us that $|\alpha|$ and $|r|$ are at most $n$ and $|\xi|$ and $|\eta|$ are at most  $n^2$. The existence of a suitable constant $B_2$ follows.}  

Assume now {that $d \geq 3$ and}   that the existence of $B_{d-1}$ has been
established.  Suppose that $u,v\in\G_d$ are conjugate and $\max\{ d(1,u),\, d(1,v)\} \le n$.

\begin{lemma}\label{easy-case}
Proposition \ref{p:upper} is true for elements $u,v\in \Ad= \Z^d\rtimes 1 <\G_d$.
\end{lemma}

\begin{proof}
We first consider the case $v\in\<a_{d-1}, a_d\>$. In this case, the conjugates of $v$ all have the form $va_d^{em}$ where $e$ is the exponent sum of $a_{d-1}$ in $v$, and
the unique shortest element conjugating $v$ to $va_d^{em}$ is $t^m$. Also, $d(1,va_d^{em})\le n$
implies $d(1, a_d^{em}) \le d(1, v) + d(1, va_d^{em}) \le 2n$, hence $|m|\le |em| \le  (2n)^d$ by Lemma \ref{l:normal-form}.

Next we suppose $v\in\Ad\smallsetminus\<a_{d-1}, a_d\>$. As $\<a_{d-1}, a_d\>$ and
 $\Ad$ are normal, $u$ also lies in $\Ad\smallsetminus\<a_{d-1}, a_d\>$.
 The unique shortest element conjugating $u$ to $v$ is again  a power of $t$.
 Proposition \ref{p:centr}(2) tells us  that this same power
 of $t$ is the unique shortest element  conjugating the image of $u$ to the image of $v$ in $\G_{d-1}$,
 and by induction this has length less than $B_{d-1}n^{d-1}$. 
 \end{proof}

\noindent \textbf{The generic case in Proposition \ref{p:upper}.} Continuing our proof of Proposition \ref{p:upper}, 
we now consider the case where the normal forms of $u$ and $v$ contain a non-zero power of $t$, in
other words $u, v\not\in\Ad$.
Let $\ov{u}$ and $\ov{v}$ be geodesic words representing the images of $u$ and $v$ in $\G_{d-1}$. Note that 
$|\ov{u}|\le d(1,u)$ and $|\ov{v}|\le d(1,v)$. Therefore, by induction, there is a word $w$ in the
free group on $\{t,a_1,\dots,a_{d-1}\}$ of length at most $B_{d-1}n^{d-1}$ such that $w^{-1}\ov{u}w = \ov{v}$ in $\G_{d-1}$.

Theorem \ref{t:dehn} provides an integer $A_{d-1}$  
such that 
$${\rm{Area}}_{\G_{d-1}}(w^{-1}\ov{u}w\ov{v}^{-1}) \le A_{d-1} |w^{-1}\ov{u}w\ov{v}^{-1}|^{d} \le A_{d-1}(2n +  {2} B_{d-1}n^{d-1})^d \le A_{d-1}'n^{d(d-1)},$$ 
where for convenience we define  $A_{d-1}':= 2^{2d}A_{d-1}B_{d-1}^d$. 
It follows from Lemma~\ref{l:dehn} that $w^{-1}{u}w = {v}a_d^\ell$ in $\G_{d}$ 
where $|\ell|\le  A_{d-1}'n^{d(d-1)}$.

Our inductive proof  will be complete if we can prove the following lemma. For then, 
in the case $v=\g$ with $A=A_{d-1}'$, we have $(wz)^{-1}{u}(wz) = {v}$ with 
$$
d(1,wz) \le |w| + d(1,z) \le  B_{d-1} n^{d-1}+ A_{d-1}' E_d n^d$$ and we can take 
$B_d = B_{d-1}+A_{d-1}' E_d$.

\begin{lemma}\label{l:last} For $d\ge 3$ and $A$ a positive integer, 
there is a constant $E_d$ such that whenever $\g\in\G_d\smallsetminus\Ad$ is conjugate to 
$\g a_d^{\ell}$ with $d(1,\g) \le n$ and $|\ell| \le  A\, n^{d(d-1)}$, 
then  there exists $z \in \Gamma_d$ such that
$z^{-1}\g z= \g a_d^{\ell}$ {and} $d(1,z)\le AE_d\, n^d$. 
\end{lemma}

The proof of the lemma involves the homomorphisms  $\zeta_g: C_{\G_{d-1}}(g)\to\Z$
defined as follows.

\begin{definition} \label{d:def}[The homomorphisms $\zeta_g$]  Let $g\in\G_{d-1}$.
For each element of the centralizer $x\in C_{\G_{d-1}}(g)$ 
and all preimages $\tilde{x}, \tilde{g} \in\G_d$
we have $\tilde{x}^{-1}\tilde{g}\tilde{x} = \tilde{g} a_d^{m}$ in $\G_d$, where $m\in\Z$ is independent of the
choices of $\tilde{x}$ and $\tilde{g}$ because different choices differ by a power of $a_d$, which is central.
Define $\zeta_g(x) :=m$ and note that $\zeta_g: C_{\G_{d-1}}(g)\to \Z$ is a homomorphism that vanishes
on $\<g\>$.
\end{definition}

\begin{proof}
The elements $c\in \G_d$ that conjugate $\gamma$ to an element of the form $\gamma a_d^{\ell}$ are precisely those whose image $\ov{c}\in\G_{d-1}$ centralise the image $\ov{\g}$ of $\g$ and have  $\zeta_{\overline{\gamma}}(\overline{c})=\ell$.  
In the light of these comments,  the lemma is a consequence of the following claim.

\medskip
\noindent{\em{\underline{Claim 1:}}}
{\em If $g\in\G_{d-1}\smallsetminus\A_{d-1}$ with $d(1,g)\le n$
and $\ell \in \im \zeta_g$ with $|\ell|\le A\, n^{d(d-1)}$,  then there exists $y\in C_{\G_{d-1}}(g)$ with $d(1,y)< AE_d\, n^d$
and $\zeta_g(y) = \ell$.}
\medskip

Before proving Claim~1, we examine $\im \zeta_g$.  
From Proposition \ref{p:centr}(1) we know  that $C_{\G_{d-1}}(g) = \<g_0, a_{d-1}\>\cong\Z^2$  where 
$$g=g_0^pa_{d-1}^r$$
with $p>0$ maximal.  A count of the letters $t$ gives $p \leq n$.  
Note that by moving the $a_{d-1}$ letters into $g_0$ if necessary, we may assume $0<r<p$.

\medskip
\noindent{\em{\underline{Claim 2:}}} 
{\em If $g=g_0^pa_{d-1}^r=t^{pq}x$ with $x\in  \A_{d-1}$  
and $0<r<p$, then $\im\, \zeta_g < \Z$ is generated by
$\zeta_g(a_{d-1}^{-1})=pq$ and $\zeta_g(g_0)=rq$. Hence $\im\zeta_g = qe\Z$, where $e=\gcd (p,r)$.}
\medskip

To see the truth of Claim 2, recall that $\zeta_g(h)$ is independent of the choices
of preimages that we take in $\G_d$, so writing $x$ as a geodesic word in the letters $a_i$,
we can take  $t^{pq}x$ 
as the preimage of $g$ in $\G_d$ and calculate 
\begin{equation}\label{e:pqorigin} 
-\zeta_g(a_{d-1})  {= \zeta_g(a^{-1}_{d-1})}=pq 
\end{equation}   as follows
$$ 
a_{d-1}t^{pq}xa_{d-1}^{-1}= t^{pq}\phi_d^{pq}(a_{d-1})xa_{d-1}^{-1}=t^{pq}a_{d-1}a_d^{pq}xa_{d-1}^{-1}=t^{pq}xa_d^{pq},
$$  
and we can verify that $\zeta_g(g_0)=qr$ by fixing a preimage $\tilde{g_0}=t^qy$ with $y\in  \A_d$ 
then defining $\tilde{g}:=\tilde{g_0}^pa_{d-1}^r$ and calculating
$$
\tilde{g_0}^{-1}(\tilde{g_0}^pa_{d-1}^r)\tilde{g_0}  = \tilde{g_0}^p \tilde{g_0}^{-1}a_{d-1}^r\tilde{g_0}
=\tilde{g_0}^p {y^{-1}} \phi_d^q(a_{d-1}^r) {y} 
= \tilde{g_0}^p {y^{-1}} a_{d-1}^ra_d^{qr} {y}  = \tilde{g} a_d^{qr}.
$$ 
This proves Claim 2.

\smallskip
With Claim 2 in hand, our next goal is to find a small element of $C_{\G_{d-1}}(g)$ whose image
 generates $\im \zeta_g$. 
Lemma \ref{l:euclid} provides us with positive integers  $\lambda < p$ and $\mu \le r<p$
such that  $\lambda  r-  \mu p =e$. 
Then, for $g=g_0^pa_{d-1}^r$ we have  
\begin{equation}\label{e:short1} 
\zeta_g(g_0^{\lambda} a_{d-1}^{\mu}) = \lambda  \zeta_g(g_0) - \mu \zeta_g(a_{d-1}^{-1}) = \lambda   rq - \mu pq = qe.
\end{equation}
Corollary~\ref{c:roots}, applied to the equality $g a_{d-1}^{-r} = g_0^p$ in $\G_{d-1}$, gives us the first of the following inequalities and $r < p$ gives the second: 
\begin{equation}\label{e:corapp} 
d(1, g_0)   \leq   K_{d-1} d(1, g a_{d-1}^{-r})   \leq   K_{d-1} (d(1, g) + p).
\end{equation}   
 So  we estimate that in $\G_{d-1}$ we have 
\begin{equation}\label{e:short}
d(1, g_0^{\lambda} a_{d-1}^\mu) \le\, \lambda  \ d(1, g_0) + \mu  
  \le {\lambda K_{d-1}  \,  (d(1,g)+p)} + \mu \le p(K_{d-1} (n+p) + 1).
\end{equation} 

\medskip

\noindent{\bf{The Final Argument.}} We are now ready to prove Claim 1. 
We have $q \neq 0$ because $g \notin \A_{d-1}$.  So,  given $\ell \in\im \zeta_g$ with $|\ell|\le A\, n^{d(d-1)}$,  because $pq \in \im \zeta_g = qe\Z$,    we  can  write 
$$
\ell = Mpq + \theta qe 
$$
for integers $\theta$ and $M$  with $0\le \theta < p$ and $|M|\le A\, n^{d(d-1)}$. Thus,  using \eqref{e:pqorigin} and \eqref{e:short1}, we have  
$$
\ell = M\, \zeta_g(a_{d-1}^{-1}) + \theta\, \zeta_g(g_0^\lambda a_{d-1}^{\mu}) = \zeta_g(a_{d-1}^{-M}(g_0^\lambda a_{d-1}^{\mu})^\theta ).
$$
We are going to argue that  the term
$$y:=   a_{d-1}^{-M} (g_0^\lambda a_{d-1}^{\mu})^\theta$$   
from the last bracket satisfies  Claim 1. 

From the triangle inequality in $\G_{d-1}$ we have, invoking  {Lemma~\ref{l:real} and} estimate (\ref{e:short}) in the second line, and using $|M|\le A\, n^{d(d-1)}$  and  $p \leq n$ in the third,   
\begin{align*}
d(1,y) = d(1, a_{d-1}^{-M}(g_0^\lambda a_{d-1}^{\mu})^\theta) &\le d(1, a_{d-1}^M) + \theta\, d(1,  g_0^\lambda a_{d-1}^{\mu})\\
&\le D_{d-1}{|M|}^{1/(d-1)} + p^2(K_{{d-1}} (n+p) + 1)\\
&\le D_{d-1}A n^d + n^2 (2K_{{d-1}}  n + 1)\\
&\le AE_d n^d
\end{align*}
provided $d\ge 3$ and we choose  $E_d$ to be sufficiently larger than $D_{d-1}$ and $K_{d-1}$.    
This completes the proof of Claim 1, hence Lemma~\ref{l:last}, Proposition~\ref{p:upper}
and Theorem~\ref{t:main-poly}.
\end{proof} 
 
\bibliographystyle{alpha}
\bibliography{bibli}

\def\cprime{$'$}
\begin{thebibliography}{MMNV22}

\bibitem[BMS93]{BMS}
G.~Baumslag, C.~F. Miller, III, and H.~Short.
\newblock Isoperimetric inequalities and the homology of groups.
\newblock {\em Invent. Math.}, 113(3):531--560, 1993.

\bibitem[BORS02]{BORS}
J.-C. Birget, A.~Yu. Ol{\cprime}shanskii, E.~Rips, and M.~V. Sapir.
\newblock Isoperimetric functions of groups and computational complexity of the
  word problem.
\newblock {\em Ann. of Math. (2)}, 156(2):467--518, 2002.

\bibitem[BP94]{BridsonPittet}
M.~R. Bridson and Ch. Pittet.
\newblock Isoperimetric inequalities for the fundamental groups of torus
  bundles over the circle.
\newblock {\em Geom. Dedicata}, 49(2):203--219, 1994.

\bibitem[BR25a]{BrRi3}
M.~R. Bridson and T.~R. Riley.
\newblock Groups with fast growing conjugator length functions.
\newblock preprint,
  \texttt{\href{https://arxiv.org/abs/2512.23674}{arXiv:2512.23674}}, 2025.

\bibitem[BR25b]{BrRi1}
M.~R. Bridson and T.~R. Riley.
\newblock Linear {D}iophantine equations and conjugator length in 2-step
  nilpotent groups.
\newblock to appear in Bull. Lond. Math. Soc.,
  \href{https://arxiv.org/abs/2506.01239}{\texttt{arXiv:2506.01239}}, 2025.

\bibitem[Bri99]{Bridson3}
M.~R. Bridson.
\newblock Fractional isoperimetric inequalities and subgroup distortion.
\newblock {\em J. of the A. Math. Soc.}, 12(4):1103--1118, 1999.

\bibitem[BRS]{BrRiSa2}
M.~R. Bridson, T.~R. Riley, and A.~Sale.
\newblock Conjugacy in a family of free-by-cyclic groups.
\newblock preprint,
  \href{https://arxiv.org/abs/2506.01248}{\texttt{arXiv:2506.01248}}.

\bibitem[BRS26]{BrRiSa}
M.~R. Bridson, T.~R. Riley, and A.~Sale.
\newblock Conjugator length in finitely presented groups.
\newblock in preparation, 2026.

\bibitem[CM77]{CollinsMiller}
D.~J. Collins and C.~F. Miller, III.
\newblock The conjugacy problem and subgroups of finite index.
\newblock {\em Proc. London Math. Soc. (3)}, 34(3):535--556, 1977.

\bibitem[Gro93]{gromov}
M.~Gromov.
\newblock Asymptotic invariants of infinite groups.
\newblock In G.~Niblo and M.~Roller, editors, {\em Geometric group theory II},
  number 182 in LMS lecture notes. Camb. Univ. Press, 1993.

\bibitem[Gui70]{Gui}
Y.~Guivarc'h.
\newblock Groupes de {L}ie \`a{} croissance polynomiale.
\newblock {\em C. R. Acad. Sci. Paris S\'er. A-B}, 271:A237--A239, 1970.

\bibitem[MMNV22]{MMNV}
J.~Macdonald, A.~Myasnikov, A.~Nikolaev, and S.~Vassileva.
\newblock Logspace and compressed-word computations in nilpotent groups.
\newblock {\em Trans. Amer. Math. Soc.}, 375(8):5425--5459, 2022.

\bibitem[Sal16]{Sale3}
A.~W. Sale.
\newblock Conjugacy length in group extensions.
\newblock {\em Comm. Algebra}, 44(2):873--897, 2016.

\end{thebibliography}
%\bibliography{$HOME/Dropbox/Bibliographies/bibli}

\bigskip

\ni \small{Martin R.\ Bridson, Mathematical Institute, Andrew Wiles Building, Oxford OX2 6GG, United Kingdom.   {bridson@maths.ox.ac.uk},   
\href{http://people.maths.ox.ac.uk/bridson/}{people.maths.ox.ac.uk/bridson/}}

\ni \small{Timothy R.\ Riley, \rule{0mm}{6mm} 
Department of Mathematics, 310 Malott Hall,  Cornell University, Ithaca, NY 14853, USA. {tim.riley@math.cornell.edu},   
\href{https://math.cornell.edu/timothy-riley}{math.cornell.edu/timothy-riley}}

 \end{document}